\documentclass[AMA,Times1COL]{WileyNJDv5}

\articletype{Research Article}%

\received{Date Month Year}
\revised{Date Month Year}
\accepted{Date Month Year}
\journal{Journal}
\volume{00}
\copyyear{2024}
\startpage{1}

\begin{document}

\title{Suboptimal Model Predictive Control with a Computation Governor: Stability, Recursive Feasibility, and Applications to Alternating Direction Method of Multipliers}

\author[1,2]{Steven van Leeuwen}

\author[1]{Ilya Kolmanovsky}

\titlemark{Suboptimal MPC with Applications to ADMM}

\address[1]{\orgdiv{Department of Aerospace Engineering}, \orgname{University of Michigan}, \orgaddress{\state{Michigan}, \country{USA}}}

\address[2]{\orgname{Johns Hopkins University Applied Physics Lab}, \orgaddress{\state{Maryland}, \country{USA}}}

\corres{Steven van Leeuwen \email{svanlee@umich.edu}}

\abstract[Abstract]{The paper considers a computational governor strategy to facilitate the implementation of Model Predictive Control (MPC) based on inexact optimization when the time available to compute the solution may be insufficient.  In the setting of linear-quadratic MPC and a class of optimizers that includes Alternating Direction Method of Multipliers (ADMM), we derive conditions on the reference command adjustment by the computational governor and on a constraint tightening strategy which ensure convergence of the modified reference command, recursive feasibility, and closed-loop stability. An online procedure to select the modified reference command and construct an implicit terminal set is also proposed. A simulation example is reported which illustrates the developed procedures.}

\keywords{Model Predictive Control, ADMM, Constraint Tightening, Suboptimality, Reference Governor}

\maketitle

\section{Introduction}
Typical implementations of Model Predictive Control (MPC) utilize the numerical solution of the underlying Optimal Control Problem (OCP). When time to compute the solution is insufficient, inexact solutions are often used; such an approach is referred to as suboptimal MPC.

In the previous conference paper \cite{SVL_ADMM_MPC}, we considered how the Alternating Direction Method of Multipliers (ADMM) can be used to implement Linear Quadratic MPC (LQ-MPC) for reference command tracking while terminating ADMM at a desired level of suboptimality and in such a way that the computation times can be reduced. The proposed strategy combined reference command modification via a Computation Governor (CG) with constraint tightening to ensure that primal infeasibility with respect to the OCP that ADMM solves does not lead to the actual constraint violations (see related works \cite{closest_to_me} \cite{main_suboptimal}). The analysis used the characterization of convergence rates \cite{optimal_step_size}. Asymptotic tracking of the constraint admissible reference commands was shown through simulation.

The previous paper \cite{SVL_ADMM_MPC} did not provide formal guarantees of asymptotic stability, nor a characterization of  regions of attraction (ROA); this is addressed in this paper. Notably, there have been related studies aimed at ADMM-type optimizers, inexact solutions, and stability guarantees. In the regulator case (i.e., when reference command tracking is not considered) the works \cite{main_suboptimal} \cite{other_suboptimal} characterized recursive feasibility and asymptotic stability properties in the setting of inexact LQ-MPC and operator splitting algorithms. In Yang et al. \cite{main_suboptimal} the Fast Alternating Minimization Algorithm was used in conjunction with dynamic constraint tightening. 

In this paper, we consider LQ-MPC which is solved numerically using the ADMM optimization algorithm or another optimization algorithm which satisfies similar properties (to be further delineated in the paper). We leverage the results in Yang et al. \cite{main_suboptimal} and exploit them in the setting of the CG.

 The CG is used to reduce number of iterations of the optimizer \cite{ipm_cg_journal} and can also act as a feasibility governor \cite{Skibik_FG} \cite{Skibik_ROA}, that expands ROA estimates of desired references \cite{Skibik_FG}. We show Input-to-State Stability (ISS) \cite{ISS_Lyapunov_original} properties as well as the convergence of the modified reference to the desired reference. We employ ISS-Lyapunov functions per each modified reference; this approach differs from the ones in the previous proofs of reference convergence \cite{Skibik_FG} \cite{Skibik_ROA} which constructed the ISS-Lyapunov functions with a dependence on the increments of the reference commands. 

In addition, we propose an online procedure for less conservatively selecting a suboptimality criterion for the optimizer leading to a further reduction of ADMM iterations. The online procedure also automates the computation of a ROA of the modified reference, which is inspired by the implicit terminal set characterization \cite{stable_no_terminal_constraint}. This paper's contributions include (i) extending the feasibility and stability analysis to the setting of inexact optimization and constraint tightening as well as general classes of optimizers, (ii) introducing a CG with reference modification for the setting, (iii) establishing convergence guarantees for the CG with piecewise ISS-Lyapunov functions, (iv.) developing an efficient online construction of the implicit terminal set and corresponding ROA.

The paper is organized as follows. Section 2 summarizes key notations. Section 3 details the tracking MPC formulation with constraint tightening. Section 4 introduces the Computation Governor. Section 5 describes the ADMM algorithm. Section 6 presents the analysis of the recursive feasiblity, asymptotic stability, and convergence properties and contains the main contributions. Section 7 treats a numerical example.

\section{Notation} For two matrices $D_1,D_2$ of appropriate size, we let $(D_1,D_2) = \begin{bmatrix}
    D_1 \\ D_2
\end{bmatrix}$. A vector of ones is denoted by $\vec{1}$. The diagonal matrix with the value $a$ on the diagonal is denoted by $I_a$ with its dimension being clear from the context. The indicator function of a set $\mathcal{C}$ is $\mathcal{I}_{\mathcal{C}}$. For a real (square) matrix $D (E)$, $a^T E a = ||a||_E^2$, and, $\overline{\lambda}(D), \lambda_{max}(E), \lambda_{min}(E)$ denote the largest singular value of $D$, largest eigenvalue of $E$, and smallest eigenvalue of $E$, respectively. The symbols $\succ(\succeq)$ denote positive (semi)definiteness. The quantity $c_{k|t}$ denotes the predicted value of a variable $c$ at the time instant $k+t$ when the prediction is made at the time instant $t$. When $c$ represents a variable updated by an iterative optimization algorithm, we use $c_{j|t}$ to denote the value of $c$ at the iteration $j$ of the algorithm.  Finally, $c_{v|t}$ denotes a steady-state quantity resulting from the reference command $v$ applied to the system at time instant $t$. The matrix $\Xi = [I \ 0 \ ... \ 0]$ is a selection matrix for $c_{0|t}$. The kronecker product is $\otimes$. The projection operation is $\Pi$.

\section{MPC Tracking Problem}
Consider a dynamic system represented by a linear model 
\begin{align}
    x_{t+1} = Ax_t +Bu_t
\end{align}
with $ t \in  \mathbb{Z}_{\geq 0}, \ x_t \in \mathbb{R}^{n_x}, \ u_t \in \mathbb{R}^{n_u}$, being the discrete time instant, states, and controls inputs, respectively. Let $Z = [A-I \ B]$, and let $G = (G_x, G_u)$ be a basis for the null space of $Z$. Then the set of steady-state equilibrium pairs $\{(x_v, u_v)\}$ can be parameterized as $ \{ Gv \ | \ v \in \mathbb{R}^{n_v} \}$, where $v$ is treated as the reference (setpoint) command. The time sequence of desired reference commands is $\{v_t\}, v_t \in \mathcal{V}$, where $\mathcal{V}$ is a compact and convex subset of $\mathbb{R}^{n_v}$. The desired equilibrium pair corresponding to $v_t$ is $(x_{v|t},u_{v|t}) = (G_x v_t, G_u v_t)$.

The LQ-MPC involves solving the following OCP at each time instant $t$, with the resulting $u_{0|t}$ applied to (1) at time $t$:
\begin{subequations}
\begin{gather}
\min_{\xi_t, \eta_t}||x_{N|t}-x_{v|t}||_P^2 + \sum_{k=0}^{N-1} ||x_{k|t}-x_{v|t}||_Q^2 +||u_{k|t}-u_{v|t}||_R^2
\end{gather}
$s.t.$
\begin{gather}
x_{k+1|t} = Ax_{k|t} + Bu_{k|t}, \ k \in \mathbb{Z}_{[0,...,N-1]} \\
u_{k|t} \in \mathcal{U}, \ k \in \mathbb{Z}_{[0,...,N-1]} \\
x_{k|t} \in \mathcal{X}, \ k \in \mathbb{Z}_{[1,...,N]}.
\end{gather}
\end{subequations}
In (2), $x_{0|t} = x_t$ is given, and (2) is parameterized by $\theta_t = (x_t,v_t)$. We assume that the state and control constraints are affine so that $\mathcal{X} \times \mathcal{U} \subset \mathbb{R}^{n_c}$, with $\mathcal{X} = \{ x \ | \ \acute{b}_x \leq C x \leq \grave{b}_x \}, \ \mathcal{U} = \{ u \ | \ \acute{b}_u \leq D u \leq \grave{b}_u \}$.

We define $\eta_t = (u_{0|t}, u_{1|t},...,u_{N-1|t}) \in \mathbb{R}^{Nn_u}, \ \xi_t = (x_{0|t}, x_{1|t},...,x_{N|t}) \in \mathbb{R}^{Nn_x}, \ \chi_t = (\eta_t,\xi_t), \ \tilde{x}_t = x_t - x_{v|t}, \ \tilde{u}_t = u_t - u_{v|t}, \ l(\tilde{x},\tilde{u}) = ||\tilde{x}||^2_Q + ||\tilde{u}||^2_R, \ F(\tilde{x}) = ||\tilde{x}||^2_P$. Furthermore, we let
 $\tilde{\eta}_t = \eta_t - u_{v|t} \otimes \vec{1}, \ \tilde{\xi}_t = \xi_t - x_{v|t} \otimes \vec{1}, \ \tilde{\chi}_t = (\tilde{\eta}_t,\tilde{\xi}_t)$.

The following assumptions are made regarding the MPC formulation:
\begin{assumption}
The set $\mathcal{X} \times \mathcal{U}$ is compact, and contains the origin in its interior. Furthermore, $(x_{v|t},u_{v|t})$ is in the interior of $\mathcal{X} \times \mathcal{U}$ for all $v_t \in \mathcal{V}$.
\end{assumption}
\begin{assumption}
The pair $(A,B)$ is stabilizable.
\end{assumption}
\begin{assumption}
The weight matrices of (2) satisfy $P,Q,R \succ 0$ and the Riccati equation to ensure local stability, $P = Q + A^T P A -(A^TPB)K$, with $K = (R + B^TPB)^{-1} (B^TPA)$.
\end{assumption}

The OCP (2) can be condensed \cite{predictive_control_linear_and_hybrid}, by eliminating the state sequence $\xi_t$, to a QP problem which has the following form,
\begin{subequations}
\begin{gather} 
\min_{\eta_t} \frac{1}{2} \eta_t^T H \eta_t + \eta_t^T W \theta_t \\
s.t. \ \underline{\underline{b}} \leq M \eta_t - L\theta_t \leq \overline{\overline{b}}.
\end{gather}
\end{subequations}

In this QP, $\underline{\underline{b}},\overline{\overline{b}} \in \mathbb{R}^{Nn_c},$ represent the bounds corresponding to the constraint set $\mathcal{X} \times \mathcal{U}$, with the bounds written in the form $\underline{\underline{b}} = (\underline{\underline{b}}_u,\underline{\underline{b}}_x) = (\acute{b}_u \otimes \vec{1}, \acute{b}_x \otimes \vec{1}), \ \overline{\overline{b}} = (\overline{\overline{b}}_u,\overline{\overline{b}}_x) = (\grave{b}_u \otimes \vec{1}, \grave{b}_x \otimes \vec{1})$. The expression $M\eta_t - L \theta_t$ is used for the constraints as it allows for use of the condensed formulation with simple constraint projections. 

\subsection{Tightened Constraints}
The OCP (3) can be rewritten in the following form, which is equivalent to (3) when the scalar parameter used for constraint tightening in (4c), $\Sigma_t$, is zero:
\begin{subequations}
\begin{gather}    
\mathcal{P}^{\Sigma_t}(\theta_t) : \min_{y_t} \frac{1}{2} y_t^T \Bar{H} y_t + y_t^T \Bar{W} \theta_t \\
s.t. \ \Bar{M} y_t = L \theta_t \\
\underline{\underline{b}} + \Sigma_t\vec{1} \leq S_s y_t \leq \overline{\overline{b}} - \Sigma_t\vec{1},
\end{gather}
\end{subequations}
where $y = \begin{bmatrix} \eta \\ s \end{bmatrix}$, $\Bar{M} = [M \ -I]$, $\Bar{W} = \begin{bmatrix} W \\ 0\end{bmatrix}$, $\Bar{H} = \begin{bmatrix}
    H & 0 \\
    0 & 0
\end{bmatrix}$, and $S_\eta = [I \ 0], \ S_s = [0 \ I]$, $S_\eta y = \eta, \ S_{s} y = s$, where $s \in \mathbb{R}^{Nn_c}$ is the slack variable for (3b), and $\underline{\underline{b}} < \overline{\overline{b}}$ elementwise. The constraints (4b)-(4c) are tightened versions of (3b) when $\Sigma_t > 0$. We assume that $\overline{\Sigma}$ is such that as long as $ 0 < \Sigma_t\leq \overline{\Sigma}$, the set described in (4c) is nonempty.

Note that $\Bar{H} \succeq 0$, thus we assume that \cite{optimal_step_size}
\begin{assumption}
    The Hessian $\Bar{H}$ is positive definite on the null space of $\Bar{M}$.
\end{assumption}

We define the following sets:
$ \mathcal{U}^{\Sigma_t} = \{u \ | \ \acute{b}_u + \Sigma_t \vec{1} \leq Du \leq \grave{b}_u - \Sigma_t \vec{1} \}, \ \mathcal{X}^{\Sigma_t} = \{x \ | \ \acute{b}_x + \Sigma_t \vec{1} \leq Cx \leq \grave{b}_x - \Sigma_t \vec{1} \}, \
\mathcal{Y} = \{y \ | \ \Bar{M} y = L \theta \}, \ \mathcal{Z}^{\Sigma_t} = \{ y \ | \ \underline{\underline{b}} + \Sigma_t\vec{1} \leq S_s y \leq \overline{\overline{b}} - \Sigma_t\vec{1} \}
$.
Then we also impose the following assumption, which states the desired references are a minimum distance away from the constraint boundary:
\begin{assumption}
    There exists a constant $\underline{\Sigma} > 0$ such that $(x_{v|t},u_{v|t}) \in \text{int}(\mathcal{X}^{\underline{\Sigma}} \times \mathcal{U}^{\underline{\Sigma}}) $ for all $ v_t \in \mathcal{V}$.
\end{assumption}

Let the set of feasible initial states $x_t$ of $\mathcal{P}^{\Sigma_t}(\cdot)$ be $\Psi(\Sigma_t)$, and note it is independent of $v_t$. Let the minimizer of (4) be $(\eta^{*\Sigma_t}(\theta_t),s^{*\Sigma_t}(\theta_t))$, with $\tilde{\eta}^{*\Sigma_t}(\theta_t) = \eta^{*\Sigma_t}(\theta_t) - u_{v|t} \otimes \vec{1}$, the controls at step $k=0$ be $ u^{*\Sigma_t}(\theta_t)$, and resulting state sequence of (4) be $\xi^{*\Sigma_t}(\theta_t)$, with $\tilde{\xi}^{*\Sigma_t}(\theta_t) = \xi^{*\Sigma_t}(\theta_t) - x_{v|t} \otimes \vec{1}$. We also let $\tilde{\chi}^{*\Sigma_t}(\theta_t) = (\tilde{\eta}^{*\Sigma_t}(\theta_t),\tilde{\xi}^{*\Sigma_t}(\theta_t))$. 

\subsection{Suboptimality}
The optimizer is modeled as the following mapping
\begin{equation}
w_{j+1|t} \leftarrow \mathcal{T}(\mathcal{P}^{\Sigma_t}(\theta_t),\{w_{j|t}\}),    
\end{equation}
where $w_{j|t}$ is the optimizer state at iteration $j$ at time $t$, and $\{w_{j|t}\}$ is the sequence of optimizer states up to iteration $j$ at time $t$. The optimizer state depends on the optimization algorithm and can include primal and dual variables, etc. We denote the optimal solution in terms of optimizer state that corresponds to a unique global minimizer of (4) (Lemma 1 \cite{SVL_ADMM_MPC}) as $w^*_t$. 

In this paper we consider $\mathcal{P}^{\Sigma_t}(\theta_t)$ solved inexactly, by means of terminating the optimizer at $\hat{j}$ iterations. Note that the evolution of the optimizer states depends on $w_{0|t}$, i.e., their initial values at the initialization, which is determined by the warm-start from the previous time instant 
\begin{equation}
    w_{0|t} \leftarrow w_{\hat{j}|t-1},
\end{equation}
with the initial values at the first time instant set to zero. We do not explicitly indicate this dependence in our notations. We also assume that the optimization algorithm is globally convergent, which is the case, for instance, for ADMM. Modifications of our approach can be developed in the case of locally convergent algorithms. We then have the following mappings:

\begin{subequations}
\begin{gather}
\chi ^{*\Sigma_t}(\theta_t) \leftarrow w^*_t \\    
\hat{\chi}^{\Sigma_t}(\theta_t) \leftarrow w_{\hat{j}|t} \\
\hat{y}^{\Sigma_t}(\theta_t) \leftarrow w_{\hat{j}|t}.
\end{gather}
\end{subequations}

We also denote the primal error residual of the optimizer $\mathcal{T}$ as $\mathsf{r}_{j|t}$, and require 
\begin{equation}
    \mathsf{r}_{\hat{j}|t} \leq \Sigma_t \implies \underline{\underline{b}} \leq M S_\eta \hat{y}^{\Sigma_t}(\theta_t) - L \theta_t \leq \overline{\overline{b}}
\end{equation}
i.e, $\mathcal{P}^0(\theta_t)$ is feasible.

We define the optimal value function of $\mathcal{P}^{\Sigma_t}(\theta_t)$ as $V^{\Sigma_t}(\theta_t)$, and we let
\begin{subequations}
\begin{gather}
\psi^{\Sigma_t}(\theta_t) = \sqrt{V^{\Sigma_t}(\theta_t)}, \\
H^\chi = I_{(Q \otimes \vec{1},P,R \otimes \vec{1})}.
\end{gather}
\end{subequations}
We note that $V^{\Sigma_t}(\theta_t)$ is equal to
$||\tilde{\chi}^{*\Sigma_t}(\theta_t)||^2_{H^{\chi}} $, and $\hat{\chi}^{\Sigma_t}(\theta_t) = (\hat{\eta}^{\Sigma_t}(\theta_t), \hat{\xi}^{\Sigma_t}(\theta_t))$ aggregates the control and state sequences computed from $w_{\hat{j}|t}$. The associated sequences $\hat{\tilde{\chi}}^{\Sigma_t}(\theta_t) = (\hat{\tilde{\eta}}^{\Sigma_t}(\theta_t), \hat{\tilde{\xi}}^{\Sigma_t}(\theta_t))$ are used to define the suboptimal value function of $\mathcal{P}^{\Sigma_t} (\theta_t)$ as $\hat{V}^{\Sigma_t}(\theta_t)$, with $\hat{\psi}^{\Sigma_t}(\theta_t) = \sqrt{\hat{V}^{\Sigma_t}(\theta_t)}$. Let the controls at step $k=0$ be $ \hat{u}^{\Sigma_t}(\theta_t)$, and the $k$th predicted state be $\hat{x}_{k}^{\Sigma_t}(\theta_t)$. 

\section{Computation Governor}
Consider a sequence of references $\{v_t\}$. The CG modifies the reference command at time $t$ from $v_t$ to $\hat{v}_t$.
The basic procedure is summarized in Algorithm 1, where $\kappa_t$ is a CG parameter which determines the modified reference, $\mathcal{E}^{\Sigma} : (\theta_t, \theta_{t-1}, \Sigma_{t-1}) \mapsto \mathbb{R}$ is a continuous function which determines $\Sigma_t$; and $\mathcal{D}^{\mathsf{r}} : \{w_{j|t}\} \mapsto \mathbb{R}$ is a continuous function which determines the algorithm termination, and satisfies the property that

\begin{equation}
     ||\hat{u}^{\Sigma_t}(\theta_t) - u^{*\Sigma_t}(\theta_t)||^2 \leq \mathcal{D}^{\mathsf{r}}(\{w_{\hat{j}|t}\}).
\end{equation}

\begin{algorithm}
\caption{for time $t \geq 1$, $0 < \kappa_t \leq 1$}\label{alg:CG}
\begin{algorithmic}[1]
\Require{$x_t,\kappa_t,v_t,\theta_{t-1},w_{j|t-1}, \mathcal{D}^\mathsf{r}$}
\State $\hat{v}_t \leftarrow \kappa_t (v_t - \hat{v}_{t-1}) + \hat{v}_{t-1}$
\State $w_{0|t} \leftarrow w_{\hat{j}|t-1}$ 
\State $\Sigma_t \leftarrow \mathcal{E}^{\Sigma}(\theta_t, \theta_{t-1},\Sigma_{t-1})$
\State $j \leftarrow 0$
\While{$ \Sigma_t^2 < \mathcal{D}^{\mathsf{r}}(\{w_{j|t}\}), \mathsf{r}_{j|t}^2$}
\State $w_{j+1|t} \leftarrow \mathcal{T}(\mathcal{P}^{\Sigma_t}(x_t,\hat{v}_t),\{w_{j|t}\})$
\EndWhile
\State $\hat{j}\leftarrow j$
\Ensure{$\{w_{j|t}\}$ }
\end{algorithmic}
\end{algorithm}

Algorithm 1 ensures the satisfaction of the desired constraints if Algorithm 1 terminates with a finite number of iterations, $\hat{j}$, or the minimizer $w_t^*$ is reached if $\Sigma_t = 0$ (Proposition 1\cite{SVL_ADMM_MPC}). In Section 6 we guarantee this property by ensuring recursive feasibility through design of $\mathcal{E}^{\Sigma}$ and $\kappa_t$. 

We denote by $x_t^{A1}$ the state at time $t$ when evolving according to Algorithm 1.

\section{ADMM Algorithm}
We first specify the optimizer model $\mathcal{T}$ and the function $\mathcal{D}^{\mathsf{r}}$ for ADMM. First, the ADMM algorithm update step is defined as
\begin{gather}
    y_{j+1|t} = E_{11}(\rho z_{j|t} - \mu_{j|t}) + (-E_{11}\Bar{W} + E_{12} L) \theta \\
    z_{j+1|t} = \Pi_{\mathcal{Z}^{\Sigma_t}}(y_{j+1|t} + (1/\rho)\mu_{j|t}) \\
    \mu_{j+1|t} = \mu_{j|t} + \rho (y_{j+1|t} - z_{j+1|t}),
\end{gather}
where $\rho$ is the step size of the ADMM update, $y,z$ are the primal separable variables, and $\mu$ is the dual variable. See (Section 3\cite{SVL_ADMM_MPC} ) for the definitions of $E_{11},E_{12}$.

    For the ADMM algorithm (11)-(13), at $\hat{j}$ we have values of $w_{\hat{j}|t}$ such that $y_{\hat{j}|t} \in \mathcal{Y}^{\Sigma_t}, \ z_{\hat{j}|t} \in \mathcal{Z}^{\Sigma_t}$, with 
    \begin{gather}
    w_{\hat{j}|t} = (z_{\hat{j}|t},\mu_{\hat{j}|t}), \\
        \mathsf{r}_{\hat{j}|t} = ||y_{\hat{j}|t} - z_{\hat{j}|t}||,
    \end{gather}
where (15) satisfies (8) (Proposition 1 \cite{SVL_ADMM_MPC}). In Algorithm 1, Line 5, in the case of ADMM,
\begin{gather}
    \mathcal{D}^{\mathsf{r}}(\{w_{j|t}\}) = \lambda_{max}(T)\mathcal{D}(\{w_{j|t}\}), \\
    \mathcal{D}(\{w_{j|t}\}) = (\gamma^{-1} - 1)^{-2}||w_{j|t} - w_{j-2|t}||_T^2, 
    \end{gather}
where $ T = I_{(\vec{1},\vec{1}/\rho^2)}$ and where $\gamma$ is the $q$-linear convergence rate \cite{optimal_step_size} parameter. From (Equation 18\cite{SVL_ADMM_MPC}), $ \forall j \geq 2, ||w_{j|t} - w^*_t||_T^2 \leq \mathcal{D}(\{w_{j|t}\})$. We then can write the suboptimality bound (10)
from $\lambda_{max}(T)||w_{\hat{j}|t} - w^*_t||_T^2 \geq ||w_{\hat{j}|t} - w^*_t||^2 \geq ||\hat{\eta}^{\Sigma_t}(\theta_t)-\eta^{*\Sigma_t}(\theta_t)||^2
    \geq ||\hat{u}^{\Sigma_t}(\theta_t)-u^{*\Sigma_t}(\theta_t)||^2$, where we extract $\hat{\eta}^{\Sigma_t}(\theta_t)$ from $w_{\hat{j}|t}$. From the previous work \cite{SVL_ADMM_MPC}, $\hat{\eta}$ is the same when extracted from either $y_{\hat{j}|t},z_{\hat{j}|t}$.

\begin{remark}
    The following analysis is valid for any optimizer of the form (5) which is convergent in $\Psi(\Sigma_t)$, and for which properties (8) and (10) hold.
\end{remark}

\section{Recursive Feasibility and Stability}

We first note the following Lipschitz continuity properties (Appendix D\cite{lipshitz_variational}). For any two values $\Sigma_{a},\Sigma_{b}$ and any two values $\theta_a, \theta_b$, assuming the optimal solution exists,
\begin{gather}
    |\psi^{*\Sigma_{a}}(\theta_a) - \psi^{*\Sigma_{a}}(\theta_b)| \leq \beta_{\chi}||\theta_a - \theta_b||, \\
    ||u^{*\Sigma_{a}}(\theta_a) - u^{*\Sigma_{b}}(\theta_a)|| \leq \phi ||\Sigma_{a} - \Sigma_{b}||,
    \end{gather}
    and for the same value of $\Sigma_a$
    \begin{gather}
    ||w^*_a - w^*_b||_T \leq \beta_w ||\theta_a - \theta_b||.
\end{gather}

 The values $\beta_{\chi} > 1, \phi, \beta_w$ can be estimated through sampling based methods or from the computed (offline) explicit solution. In what follows we first discuss the constant reference case and highlight recursive feasibility and stability properties. We then address the varying reference case.
\subsection{Constant Reference}
\subsubsection{Recursive Feasibility}

For a constant reference $v \in \mathcal{V}$ with associated equilibrium pair $(x_v,u_v)$ and associated constraint tightening parameter $\Sigma' \geq 0$, in the exact optimizer setting, given a state $x_{t}$, the closed-loop system evolves according to
\begin{equation} 
x_{t+1}^\circ = Ax_{t} + Bu^{*\Sigma'}(x_t,v).
\end{equation}
A region of attraction (ROA) estimate for (21) is given by (Equation 17 \cite{main_suboptimal}):
\begin{equation}
    \Gamma^{\Sigma'}(v) = \{x \in \mathbb{R}^{n_x} \ | \ \psi^{\Sigma'}(\theta) \leq \sqrt{dp^{\Sigma'}_{v}}\}
\end{equation}
where $d = N \lambda_{min}(Q)/\lambda_{max}(P) + 1$ and $p^{\Sigma'}_{v} > 0$ is such that:
\begin{equation}
    \Omega^{\Sigma'}(v) = \{x \ | \ F(\tilde{x}) \leq p^{\Sigma'}_{v}\} \subseteq \{ x \ | \ -K\tilde{x} + u_{v} \in \mathcal{U}^{\Sigma'}, \ x \in \mathcal{X}^{\Sigma'} \}. 
 \end{equation}
\begin{remark}
    The value $p^{0}_{v}$ can always be chosen such that $p^{\Sigma'}_{v} \leq p^{0}_{v}$, thus $\Gamma^{\Sigma'}(v) \subseteq \Gamma(v)^0 \subseteq \Psi(0)$ and also $\Gamma^{\Sigma'}(v) \subseteq \Psi(\Sigma') \subseteq \Psi(0)$. 
\end{remark}

Then with Assumptions 1-3, if follows from (Theorem 1 Proof \cite{stable_no_terminal_constraint}),
\begin{equation}
  V^{\Sigma'}(x^{\circ}_{t+1},v) \leq V^{\Sigma'}(x_t,v) - ||\tilde{x}_{t}||_Q^2  
\end{equation}
for $x_{t} \in \Gamma^{\Sigma'}(v)$.

To show recursive feasibility of $x_t^{A1}$, i.e. when inexactness in the optimizer must be accounted for, we construct an invariant set under $x_t^{A1}$ by considering two subsystems. First we define a set $\check{\Gamma}^{\Sigma}(v)$ and constraint tightening parameter $\check{\Sigma}_v$ where (Equation 21 \cite{main_suboptimal})
\begin{gather}
    \check{\Gamma}^{\Sigma}(v) = \{x \ | \ \psi^0(\theta) \leq \check{\psi}_v \} \subseteq \Gamma^{\check{\Sigma}}(v) \\
    \check{\Delta}^{\Sigma}(v) = \{ \Sigma \ | \ \Sigma \leq \check{\Sigma}_v \}
\end{gather}
and where the value of $\check{\psi}_v$ is chosen such that $\check{\psi}_v \leq \sqrt{dp^{\check{\Sigma}}_{v}}$.

We next define the following inequalities for the value function $\psi^0(\theta_t)$ and constraint tightening $\Sigma_t$ which will be used to show $\psi^0(\theta_t) \leq \check{\psi}_v, \ \Sigma_t \leq \check{\Sigma}_v$.
\begin{lemma}
\emph{Suppose Assumptions 1-4 hold, $(x_t, \Sigma_t) \in \check{\Gamma}^{\Sigma}(v) \times \check{\Delta}^{\Sigma}(v)$, $v_t = v \ \forall t$, and the constraint tightening update is as follows:
\begin{equation}
    \Sigma_{t+1} = \mathcal{E}^{\Sigma}(\theta_{t+1}, \theta_t, \Sigma_t) = \pi_1 \Sigma_t + \pi_2||\theta_{t+1} - \theta_t||,
\end{equation}
where $0 < \pi_1, \pi_2 < 1$ are constants. Then the following inequalities hold:}
\begin{gather}
    \psi^0(\theta_{t+1}) \leq (1-\alpha_1)\psi^0(\theta_t) + \zeta_1\Sigma_t \\
    \Sigma_{t+1} \leq (1 - \alpha_2)\Sigma_t + \zeta_2 \psi^0(\theta_t),
\end{gather}
\end{lemma}
\emph{where $\alpha_1,\alpha_2, \zeta_1,\zeta_2$ are constants listed in Appendix A.}

\begin{proof}
The proof follows (Proof of Lemma 2\cite{main_suboptimal}). As we subsequently rely on derivations in the proof and for completeness, we include the proof in Appendix A.
\end{proof}

The following assumption facilitates the recursive feasibility property that is stated in Lemma 2: 
\begin{assumption}
 In (28), (29), $\alpha_1, \alpha_2 < 1$, $\zeta_2/\alpha_2 < \alpha_1 / \zeta_1$, and $\alpha_1 \sqrt{d} < 1$. 
\end{assumption}
\begin{remark}
    The constant $\alpha_1$ can be made to satisfy $\alpha_1 < 1$ by tuning $Q$, and $\alpha_2$ satisfies $\alpha_2< 1$ for small values of $\pi_2$. The relations $\zeta_2/\alpha_2 < \alpha_1 / \zeta_1, \ \alpha_1 \sqrt{d} < 1$ can be satisfied with appropriate values of $\beta_{\chi},\phi$, which restricts the OCP (2) design.
\end{remark}
    
Lemma 1 holds given $x_t \in \check{\Gamma}^{\Sigma}(v), \ \Sigma_t \in \check{\Delta}^{\Sigma}(v)$, thus we have the following Lemma to show (28), (29) hold for $\{x_t^{A1}\}$.
\begin{lemma} (Proposition 1\cite{main_suboptimal})
\emph{Consider $(x_t,\Sigma_t) \in \check{\Gamma}^{\Sigma}(v) \times \check{\Delta}^{\Sigma}(v),\ v_t=v \ \forall t$. Suppose Assumptions 1-4 hold, then if $\check{\psi}_v, \check{\Sigma}_v$ satisfy
\begin{equation}
    (\zeta_2/\alpha_2)\check{\psi}_v \leq \check{\Sigma}_v \leq (\alpha_1/\zeta_1)\check{\psi}_v
\end{equation}
then $\mathcal{P}^{\Sigma_t}(x_t,v)$ is recursively feasible and $\check{\Gamma}^{\Sigma}(v) \times \check{\Delta}^{\Sigma}(v)$ is an invariant set for $\{x_t^{A1}\}$}.
\end{lemma}
\begin{proof}
We use the induction argument as in (Proof of Proposition 1 \cite{main_suboptimal}) to show $(x_t,\Sigma_t) \in \check{\Gamma}^{\Sigma}(v) \times \check{\Delta}^{\Sigma}(v) \implies (x_{t+1},\Sigma_{t+1}) \in \check{\Gamma}^{\Sigma}(v) \times \check{\Delta}^{\Sigma}(v)$ and is omitted here. We make use of the fact that from Proposition 1 and the termination criteria in Algorithm 1, $\mathcal{P}^0(x_t,v)$ is feasible.
\end{proof}

\subsubsection{Asymptotic Stability}
Consider again $(x_t,\Sigma_t) \in \check{\Gamma}^{\Sigma}(v) \times \check{\Delta}^{\Sigma}(v), v_t=v \ \forall t$. We note the flow $\{x_t^{A1}\}$ can be written as \textbf{Subsys 1} in the following two subsystems
\begin{subequations}
\begin{gather}
    \textbf{Subsys 1:} \ x_{t+1} = Ax_t + B\hat{u}^{\Sigma_t}(\theta_t) \\
    \textbf{Subsys 2:} \ \Sigma_{t+1} = \mathcal{E}^{\Sigma}(\theta_{t+1}, \theta_t,\Sigma_t)
\end{gather}
\end{subequations}

where $\Sigma_t$ and $(\theta_{t+1}, \theta_t)$ are inputs for subsystems (31a), (31b) respectively, and where the values $\psi^0(\theta_t)$ and $\Sigma_t$ can be shown to be ISS-Lyapunov functions (Equations 2,3 \cite{ISS_Lypaunov},) for the subsystems (31a), (31b) respectively. Then the small gain theorem can be used to show $x_t^{A1} \rightarrow x_{v|t}$, i.e. asymptotic stability with ROA estimate $\check{\Gamma}^\Sigma(v)$. The small gain theorem proof for (31) is in (Proof of Theorem 2\cite{main_suboptimal}).

\subsection{Varying References}
In this section we consider modified reference commands $\hat{v}_t$ generated from Algorithm 1, Line 1. We show that with suitable choices of $\{\kappa_t\}, \ x_{t}^{A1} \in \check{\Gamma}^{\Sigma}(\hat{v}_{t}) \implies x_{t+1}^{A1} \in \check{\Gamma}^{\Sigma}(\hat{v}_{t+1})$, and we also demonstrate reference convergence, $\hat{v}_i \rightarrow v_t, \ i \geq t$ if $ \ v_{t+1} = v_t \ \forall t$. 

We first note the following property of the sequence of references $\{v_t\}$, which follows from convexity of $\mathcal{V}$.
\begin{proposition}
    \emph{Suppose that $v_a,v_b \in \{v_t\}$ then $v' = v_a + \kappa (v_b - v_a) \in \mathcal{V}$ for $0 \leq \kappa \leq 1$.}
\end{proposition}

From (18), for $\theta_c = (x,v_c), \theta_d=(x,v_d)$ it follows that:
\begin{equation}
|\psi^0(\theta_c) - \psi^0(\theta_d)| \leq \beta_{\chi}||v_c - v_d||.
\end{equation}
Furthermore we have the following:
\begin{equation}
\hat{\psi}^{\Sigma}(\theta_t) \leq \beta_{\chi}\epsilon \implies x_i \in \Psi (\underline{\Sigma}), \ \forall v \in \mathcal{V}, \ i \geq t    
\end{equation}
where $\epsilon > 0$ is a small constant, the existence of which follows from Assumption 5. We now describe how to choose $\kappa_t$.

\subsubsection{Choice of $\kappa_t$}
To ensure $x_{t+1}^{A1} \in \check{\Gamma}^{\Sigma}(\hat{v}_{t+1})$, $\kappa_{t+1}$ must be chosen appropriately.

We first note an upper bound on $\psi^0(\theta_{t+1})$ as
 \begin{equation}
     \psi^0(x_{t+1},\hat{v}_{t+1}) \leq 
      |\psi^0(x_{t+1},\hat{v}_{t+1}) -  \psi^0(x_{t+1},\hat{v}_t)| + \psi^0(x_{t+1},\hat{v}_t)
     \leq \beta_{\chi}||\hat{v}_{t+1} - \hat{v}_t|| + \psi^{0}(x_{t},\hat{v}_t) + \zeta_1 \Sigma_t
 \end{equation}
 where we have used (32), that $\psi^0(x_{t+1},\hat{v}_t) \leq \psi^0(x^\circ_{t+1},\hat{v}_t) + \big| \psi^0(x_{t+1},\hat{v}_t) - \psi^0(x^\circ_{t+1},\hat{v}_t) \big| \leq  \psi^0(x^\circ_{t+1},\hat{v}_t) + \zeta_1 \Sigma_t$ from (A3), and then used (24).

We modify the definition of $\mathcal{E}^{\Sigma}$  in (27):
\begin{equation}
    \mathcal{E}^{\Sigma}(\theta_{t+1}, \theta_{t}, \Sigma_t) = \check{\Sigma}_{\hat{v}_{t+1}} \ \textbf{if} \ \kappa_{t+1} > 0.
\end{equation}

Then it is possible to ensure $x_{t+1}^{A1} \in \check{\Gamma}^{\Sigma}(\hat{v}_{t+1})$ at time instant $t+1$ if $\kappa_{t+1}>0$ and the following conditions hold:
\begin{subequations}
    \begin{gather}
     \beta_{\chi}||\hat{v}_{t+1} - \hat{v}_t|| + \psi^0(x_{t},\hat{v}_{t}) + \zeta_1 \Sigma_t \leq \check{\psi}_{\hat{v}_{t+1}} \leq  \sqrt{dp^{\Sigma_{t+1}}_{\hat{v}_{t+1}}} \\
    (\zeta_1/\alpha_1) \Sigma_{t+1} \leq \check{\psi}_{\hat{v}_{t+1}} \leq (\alpha_2/\zeta_2) \Sigma_{t+1}.  
 \end{gather}
 \end{subequations}
 
 The following theorem shows that the inequalities (36) can be satisfied with the choice of $\kappa_{t+1} = \epsilon$ if over a sufficiently long time window there is no modification of reference.
 
 \begin{theorem}
     \emph{Let $\epsilon$ satisfy (33). Consider references $\{\hat{v}_t\}$ from Algorithm 1, where $v_{t+1}= v_t \ \forall t$, and suppose at time instant $i'$, $||\hat{v}_{i'} - v_t|| > \epsilon$. 
     Then there exists a time instant $i > i'$ and $\hat{v}_i \in \mathcal{V}$ such that if $\kappa_t = 0, \ i' \leq t \leq i-1$ then $||\hat{v}_i - \hat{v}_{i-1}|| = \epsilon$ and (36) is satisfied (where $t+1=i$)}.
 \end{theorem}
 \begin{proof}
     To satisfy Theorem 1 we find values for $\Sigma_i, \check{\psi}_{\hat{v}_i},p^{\Sigma_i}_{\hat{v}_i}$ and $\epsilon$ for which (36) holds when $i > i'$ is sufficiently large.
     
     We first set
     \begin{equation}
         \check{\psi}_{\hat{v}_i} \leftarrow (\zeta_1/\alpha_1)\Sigma_i
     \end{equation} which satisfies (36b) from Assumption 6. We then let the time instants progress until qualifications Q1, Q2, and Q3 that follow are met and show (36a) is also satisfied.
     
     Consider the first inequality in (36a). From asymptotic convergence of the state for a constant reference, it follows that for sufficiently large $i$ the condition 
     \begin{equation}
      \text{Q1:} \ \hat{\psi}^{\Sigma_{i-1}}(x_{i-1},\hat{v}_{i-1}) \leq \min(\beta_{\chi}\epsilon,\ (\sqrt{d}-1)\beta_{\chi} \epsilon)
     \end{equation}
     will be satisfied. Then with the optimality property $\psi^0(x_{i-1},\hat{v}_{i-1}) \leq \hat{\psi}^{\Sigma_{i-1}}(x_{i-1},\hat{v}_{i-1})$ we satisfy the first inequality in (36a) if
     \begin{equation}
      \sqrt{d}\beta_{\chi}\epsilon + \zeta_1 \Sigma_{i-1} \leq \check{\psi}_{\hat{v}_{i}}.
     \end{equation}
     The inequality (39) will be shown at the end of the proof.

     To consider the second inequality in (36a), which allows the construction of $\check{\Gamma}^{\Sigma}(\hat{v}_i)$, we first construct $\Omega^{\check{\Sigma}}(\hat{v}_i)$ for use in defining $\Gamma^{\check{\Sigma}}(\hat{v}_i)$, see (22). Then from (25), $\check{\Gamma}^{\Sigma}(\hat{v}_i) \subseteq \Gamma^{\check{\Sigma}}(\hat{v}_i)$. This requires us to consider the case when $\mathcal{P}^{\Sigma_i}(x_i,\hat{v}_i)$ is solved exactly,
    \begin{gather}
         F(\tilde{x}^{*\Sigma_i}_{N}(x_i,\hat{v}_{i})) \leq V^{\Sigma_{i}}(x_i,\hat{v}_{i}) \nonumber \\
    \leq V^{\Sigma_{i-1}}(x_i,\hat{v}_{i-1}) \nonumber \\
    \leq (\psi^{\Sigma_{i-1}}(x_{i-1},\hat{v}_{i-1}) + \zeta_1 \Sigma_{i-1})^2 
    \end{gather}
    where for the third inequality we have used $\psi^0(x_{i},\hat{v}_{i-1}) \leq \psi^0(x^\circ_{i},\hat{v}_{i-1}) + \big| \psi^0(x_{i},\hat{v}_{i-1}) - \psi^0(x^\circ_{i},\hat{v}_{i-1}) \big| \leq  \psi^0(x_{i-1},\hat{v}_{i-1}) + \zeta_1 \Sigma_{i-1}$ as in (34), then $\psi^0(x_{i-1},\hat{v}_{i-1}) \leq \psi^{\Sigma_{i-1}}(x_{i-1},\hat{v}_{i-1})$. The second inequality necessitates that the sequence $\chi^{*\Sigma_{i-1}}(x_i,\hat{v}_{i-1})$ be admissible for $\mathcal{P}^{\Sigma_i}(x_i,\hat{v}_{i})$. To show this, we necessitate at time instant $i$ that 
    \begin{equation*}
    \text{Q2:} \ \Sigma_{i-1} < \underline{\Sigma}.
    \end{equation*}
     Note that Q2 is always satisfied for $i$ sufficiently large (see Section 6.1.2). Then from  $\hat{\psi}^{\Sigma_{i-1}}(x_{i},\hat{v}_{i-1}) \leq \hat{\psi}^{ \Sigma_{i-1}}(x_{i-1},\hat{v}_{i-1}) \leq \beta_{\chi}\epsilon$, where we used (33) from (38), we have $x_i \in \Psi(\underline{\Sigma})$. Then we choose $\Sigma_i$ such that $ 0 < \Sigma_{i-1} <\Sigma_i < \underline{\Sigma}$ and conclude $\Psi(\underline{\Sigma}) \subset \Psi(\Sigma_i) $. 
    
    Then given an upper bound on $F(\tilde{x}^{*\Sigma_i}_{N}(x_i,\hat{v}_{i}))$ in (40), consistent with (23) we set $p$ at time instant $i$ as 
    \begin{gather}
        p^{\Sigma_i}_{\hat{v}_i} \leftarrow (\beta_{\chi} \epsilon + \zeta_1 \Sigma_{i-1})^2 \\
        \check{\psi}_{\hat{v}_i} \leftarrow\sqrt{dp^{\Sigma_i}_{\hat{v}_i}}
    \end{gather}
     which satisfies the second inequality in (36a). Then from $\check{\psi}_{\hat{v}_i} = (\zeta_1/\alpha_1)\Sigma_i$, we note
        \begin{equation}
            \epsilon = \frac{1}{\beta_{\chi}}(\frac{\zeta_1 \Sigma_i}{\alpha_1 \sqrt{d}} - \zeta_1 \Sigma_{i-1}).       
        \end{equation} 
        
         To show constraints are satisfied in (23) from the construction of $\Omega^{\check{\Sigma}}(\hat v_i)$, we also necessitate at time instant $i$ that 
         \begin{equation*}
             \text{Q3:} \ \sqrt{ \frac{p^{\Sigma_i}_{\hat{v}_i}}{\lambda_{min}(P)}}, \sqrt{ \frac{\lambda_{max}(K)p^{\Sigma_i}_{\hat{v}_i}}{\lambda_{min}(P)}} \leq \underline{\Sigma} - \Sigma_i.
         \end{equation*}
         
         What remains is to show (39). Plugging in (43) into (39), we have $\sqrt{d}\beta_{\chi}\epsilon + \zeta_1 \Sigma_{i-1} =  (\zeta_1/\alpha_1) \Sigma_i - (\sqrt{d}-1)\zeta_1 \Sigma_{i-1} < (\zeta_1/\alpha_1)\Sigma_i$ and the conditions (36) are satisfied.
 \end{proof}

\subsubsection{Computation Governor and Online Procedure}

Considering the criteria in (36) for the selection of $\hat{v}_{t+1}$ needed at time instant $t+1$, if a sufficient amount of time has passed with no modification of reference, $\Sigma_{t+1}$ can be chosen, along with $||\hat{v}_{t+1} - \hat{v}_t|| = \epsilon$, according to Theorem 1, but this represents a conservative approach since $\mathcal{P}^{\Sigma_t}(\theta_t)$ would be very similar to the optimal solve on account of the small constraint tightening parameter. We would like to maximize $\Sigma_t$, subject to (35), to allow for a less restrictive suboptimality termination criterion, as well as maximize $\check{\psi}_{\hat{v}_t}$ subject to (36), to allow for a larger ROA of the modified reference and a large choice of $\kappa_{t+1}$. 

A conceptual way to approach this is to solve a multi-objective optimization problem, such as computing a Pareto-front, but we look to an efficient online approach by first maximizing $\kappa_{t+1}$ and then maximizing $\Sigma_{t+1}$ via the solution of two linear programs. These linear programs have relatively small computation times compared to the ADMM solve due to their small problem size.

The choice of $\kappa_{t+1}$ is subject to the CG strategy to bound the suboptimality of the modified reference, since a selection of a large value of $\kappa_{t+1}$ can result in a violation of (36), and even if not, high computation times of the optimizer. First, consider an upper bound $\Lambda_{t+1}$ on the suboptimality criterion $||w_{0|t+1} - w^*_{t+1}||_T^2$, i.e. $||w_{0|t+1} - w^*_{t+1}||_T^2 \leq \Lambda_{t+1}$. We would like to ensure $\Lambda_{t+1} \leq \overline{\Lambda}$, where $\overline{\Lambda}$ is a threshold for which any suboptimality $\Lambda_{t+1} > \overline{\Lambda}$ requires not adjusting the reference. We can accomplish this from writing the upper bound as (Equation 27 \cite{SVL_ADMM_MPC}), 
\begin{equation}
\Lambda_{t+1} = \sqrt{\mathcal{D}(\{w_{j|t}\})} + \beta_w ||(x_{t+1}-x_t,\kappa_{t+1}(v_{t+1}-\hat{v}_t))||.    
\end{equation}
We choose a design parameter $0 < \sigma < 1$ used in step ii.a) of the Online Procedure below which will scale down the desired suboptimality threshold at certain time instants. This is needed since a high suboptimality threshold could lead to violation of (36), as discussed. We also define a lower threshold $\underline{\Lambda}$ where if $\Lambda_{t+1} \leq \underline{\Lambda}$, reference modification is rejected.

In (43), if we choose $\Sigma_{i-1}$ and $\Sigma_i$ consistent with Theorem 1 offline, and where we denote $\Sigma_{i-1} = \omega\underline{\Sigma}$ for some $\omega$ less than 1, then we lower bound $\epsilon$ as
\begin{equation*}
\underline{\epsilon} = \frac{\omega \underline{\Sigma}}{\beta_{\chi}}(\frac{\zeta_1 }{\alpha_1 \sqrt{d}} - \zeta_1)   
\end{equation*}
for use in the Online Procedure. Then we enforce Theorem 1 Q1 and Q2 online and Q3 will be satisfied from the Online Procedure. We note the selection of $\omega\underline{\Sigma}$ holds for all references.

The following assumption is used to allow for linear program solves in the Online Procedure.

\begin{assumption}
    The constraint sets $\mathcal{U, X}$ are hyper-rectangles.
\end{assumption}

Then we are ready to state the Online Procedure.

\paragraph*{Online Procedure}

\begin{enumerate}
\item[i.] Solve the following Linear Program for an upper bound on the constraint tightening based on current state: $\Sigma_{t+1}, \eta_{t+1}$: 
$\max \Sigma'_{t+1} \ s.t. \ \begin{bmatrix}
    1 & 0\\
    -1 & 0\\
    1 & M\\
    1 & -M
\end{bmatrix} \begin{bmatrix}
    {\Sigma}'_{t+1} \\
    {\eta}'_{t+1}
\end{bmatrix} \leq \begin{bmatrix}
    \overline{\Sigma} \\
    0 \\
     L (x_{t+1},0) + \overline{\overline{b}} \\
     - L (x_{t+1},0) - \underline{\underline{b}}
\end{bmatrix} $. Note there is no dependence on the reference selection.
\item[ii.] \emph{Computation Governor}. 
\begin{enumerate}
    \item If $\kappa_t = 0$, assign $\Lambda_{t+1} = \max(\underline{\Lambda},\sigma \Lambda_t)$. Otherwise, assign $\Lambda_{t+1} \leftarrow \overline{\Lambda}$. 
    \item If $\Lambda_{t+1} = \underline{\Lambda}$, Theorem 1 Q1 and Q2 hold, and $v_{t+1} \neq \hat{v}_t$, set $\kappa_{t+1} \leftarrow \underline{\epsilon}/||v_{t+1}-\hat{v}_t||$, compute $\hat{v}_{t+1}$ and \textbf{go to} step ii.f). 
    \item If $\sqrt{\mathcal{D}(\{w_{j|t}\})} + \beta_w||x_{t+1} - x_t|| \nleq \Lambda_{t+1}$, set $\kappa_{t+1} \leftarrow 0$ and \textbf{break}.    \item $\kappa_{t+1} \leftarrow \min(1,\max(\kappa_{t+1} \ | \ (44)\}))$ and compute $\hat{v}_{t+1}$.
    \item If $0 < \kappa_{t+1} < \underline{\epsilon}/||v_{t+1}-\hat{v}_t||$, set $\kappa_{t+1} \leftarrow 0$ and \textbf{break}. 
    \item Solve for the smallest upper bound on the constraint tightening: $\Sigma''_{t+1} \leftarrow \min(\Sigma'_{t+1},\Sigma_{t+1}^{''x},\Sigma_{t+1}^{''u})$, where $\Sigma_{t+1}^{''x} = \min(\min(\overline{\overline{b}}_x - \underline{\Sigma} \vec{1} - x_{\hat{v}|t+1}), \min(x_{\hat{v}|t+1} -(\underline{\underline{b}}_x + \underline{\Sigma}\vec{1}))), \ \Sigma_{t+1}^{''u} = \min(\min(\overline{\overline{b}}_u - \underline{\Sigma} \vec{1} - u_{\hat{v}|t+1})), \min(u_{\hat{v}|t+1} - (\underline{\underline{b}}_u + \underline{\Sigma}\vec{1})))$.
\end{enumerate}

\item[iii.]
    Solve the following Linear Program for terminal set construction with $\Sigma_{t+1}, \ \overline{\overline{x}}_{t+1}$, where $\overline{\overline{x}}_{t+1}$ is the bound  where $||\tilde{x}_{t+1}|| \leq \overline{\overline{x}}_{t+1} \implies x_{t+1} \in \mathcal{X}^{\Sigma''_{t+1}}, u_{t+1} \in \mathcal{U}^{\Sigma''_{t+1}} $ :
    $\max \Sigma_{t+1} \ s.t.$ 
    $\begin{bmatrix}
    1 & 0 \\
    -1 & 0 \\
    1 & \overline{\lambda}(K) \\
    1 & 1 \\
    \zeta_1 / \alpha_1 & - \sqrt{d \lambda_{min}(P)} \\
    -\alpha_2 / \zeta_2 & \sqrt{d \lambda_{min}(P)} \\
\end{bmatrix} \begin{bmatrix}
    {{\Sigma}_{t+1}} \\
    \overline{\overline{x}}_{t+1}
\end{bmatrix} \leq \begin{bmatrix}
    {\Sigma''_{t+1}} \\
    0 \\
    {b_{t+1}^{''u}} \\
    {b_{t+1}^{''x}} \\
    0 \\ 
    0
\end{bmatrix} $
where $b_{t+1}^{''x} = \min(\min(\overline{\overline{b}}_x - x_{\hat{v}|t+1}), \min(x_{\hat{v}|t+1} -\underline{\underline{b}}_x )), \ b_{t+1}^{''u} = \min(\min(\overline{\overline{b}}_u - u_{\hat{v}|t+1})), \min(u_{\hat{v}|t+1} - \underline{\underline{b}}_u))$. Assign $p^{\Sigma_{t+1}}_{\hat{v}_{t+1}} = \lambda_{min}(P)\overline{\overline{x}}_{t+1}^2$ and $\check{\psi}_{\hat{v}_{t+1}} = \sqrt{dp^{\Sigma_{t+1}}_{\hat{v}_{t+1}}}$. Then (36b), and the second inequality in (36a) are satisfied. 
\item[iv.] 
    If $\hat{\psi}^{\Sigma_t}(x_t,\hat{v}_t) + \beta_{\chi}||\hat{v}_{t+1} - \hat{v}_t|| + \zeta_1 \Sigma_t \nleq \sqrt{dp^{\Sigma_{t+1}}_{\hat{v}_{t+1}}}$, then set $\kappa_{t+1} \leftarrow 0$ and \textbf{break}.
The first inequality in (36a) is satisfied. Note we have used the property $\psi^{0}(x_t,\hat{v}_t) \leq \hat{\psi}^{\Sigma_t}(x_t,\hat{v}_t)$.
\item[v.] \textbf{do} Algorithm 1
\end{enumerate}

    \begin{remark}
        The value of $\sigma$ should not be chosen too large, since $\kappa_{t+1}$ will evaluate to 1 in step ii.d), then rejected in iv). Additionally if it is chosen too small, step ii.c) will set $\kappa_{t+1}$ to 0.
    \end{remark}
    \begin{remark}
        In the absence of Assumption 7, the general calculation of $\Sigma_{t+1}^{'x},\Sigma_{t+1}^{'u}$ is $\Sigma_{t+1}^{''x} = \min(||\overline{\overline{b}}_x - \underline{\Sigma} \vec{1} - x_{\hat{v}|t+1}||, ||x_{\hat{v}|t+1} -(\underline{\underline{b}}_x + \underline{\Sigma}\vec{1})||), \ \Sigma_{t+1}^{''u} = \min(||\overline{\overline{b}}_u - \underline{\Sigma} \vec{1} - u_{\hat{v}|t+1})||, ||u_{\hat{v}|t+1} - (\underline{\underline{b}}_u + \underline{\Sigma}\vec{1})||)$.
    \end{remark}

\subsubsection{Modified Reference Convergence}
From the above procedure, we recover the choice of $||\hat{v}_{t+1} - \hat{v}_t|| = \epsilon$ after a finite amount of steps of $\kappa = 0$ from Theorem 1. This follows from step ii.b). and ii.e). From repeated application of $||\hat{v}_{t+1} - \hat{v}_{t}|| = \epsilon$, we have $||\hat{v}_{i} - v_t|| \leq \epsilon$ at some time instant $i$. It follows that the value $\kappa_t = 1$ becomes admissible for some time $t > i$ from the arguments of Theorem 1. When $\kappa_t = 1$, subsequent time instants use the update (27).

\section{Numerical Example}
The purpose of this section is to illustrate the Online Procedure with a numerical example, and show with the theoretical guarantees the computational burden of the optimizer remains limited. The example dynamic system is modeled after a point mass system with unstable equilibria of the form $(\cdot,0)$, with the discrete dynamics given as 
\begin{equation*}
    x = \begin{bmatrix}
        \tau \\ \dot{\tau}
    \end{bmatrix}, \
    A = \begin{bmatrix}
        1 & 0.3 \\ 0.01 & 1
    \end{bmatrix}, \
    B = \begin{bmatrix}
        0 \\ 0.01
    \end{bmatrix},
\end{equation*} where $\tau$ denotes position, and where the time step interval is 0.3. The trajectory of $x_t^{A1}$ when computing $\kappa_t$ with the Online Procedure is shown in Figure 1 in blue. In Figure 1, $x_0 = (0.194, 0), \overline{\overline{b}}_x = -\underline{\underline{b}}_x = (0.2,0.002), \ \overline{\overline{b}}_u = -\underline{\underline{b}}_u = 1$, and until instant $t = 25, \ v_t = -0.2744$ corresponding to $x_{\hat{v}|t} = (0.194,0)$, thereafter $v_t = -0.2814$ corresponding to $x_{\hat{v}|t} = (0.199,0)$. 

There are two cases $x_t^{A1}$ is compared against. These are Case 2 and Case 3, and they use exact solve settings (they correspond to ADMM run to completion). Case 2 assigns $\hat{v}_t = v_t$ and is plotted in dark grey. Case 3 assigns $\hat{v}_t$ equal to the modified reference from Algorithm 1 ($A1$), and is plotted in lighter grey. 

In the left plot in Figure 1, the modified reference behavior is clearly seen with $x_t^{A1}$. The behavior to select $\underline{\epsilon}$ in the Online Procedure as a reference modification never occurs. The light grey trajectory tracks the same modified references and has a negligible difference with the blue trajectory, whereas the dark grey trajectory rides the desired constraints. An outer boundary of the terminal set $\Omega^{\check{\Sigma}}(\hat{v}_t)$ for $A1$ at time instances of modified reference is also plotted in green.

The leftmost subplot in the upper right shows the reference modification parameter, and finite time convergence to the desired reference. The rightmost plot in the upper right shows the constraint tightening parameter, and this can also be visualized by looking at the distance from maximal radius of the terminal set per modified reference to desired constraint boundary in the left plot. The constraint tightening parameter is close to $3 \times 10^{-4}$ during time instants before when $\kappa_t = 1$, so despite $\mathcal{E}^{\Sigma}$ decaying $\Sigma_t$ which would necessitate high iterations needed in ADMM, the Online Procedure updates the reference often enough to not observe this behavior.

In the lower right plot, the relative difference in ADMM iterations is shown, and we see $A1$ outperforms Case 2 and Case 3 by looking at area under the curves, as well as in clock times and average iterations (see caption). The CG warm-start and $\Sigma_t$ termination criteria both limit the iterations needed in $A1$. 

The parameters $\gamma,\phi,\beta_\chi,\beta_w$ are chosen through sampling based methods, and are primarily a function of the OCP, see Remark 3. Conservatism (small step sizes and small constraint tightening parameter) of Algorithm 1 and the Online Procedure is offset by observing lower values of these parameters. Lower values of $\zeta_1/\alpha_1$, which are functions of $\gamma,\phi,\beta_\chi,\beta_w$, are desirable as it results in larger admissible constraint tightening selections $\Sigma_t$, which in turn lower necessary optimizer iterations. Lower values of $\gamma,\beta_\chi,\beta_w$, result in step ii.c) and step iv.) evaluating to $\kappa_{t+1} \neq 0$ more often, and are desirable. 

To exhibit reference convergence to the desired references in the smallest amount of time, given $\gamma,\phi,\beta_\chi,\beta_w$, we tuned $\overline{\Lambda},\pi_1, \pi_2, \sigma$. Our process was to fix $\overline{\Lambda}$, then keep $\pi_1$ close to $1$ and $\pi_2$ close to $0$ while tuning $\sigma$ according to Remark 4.

\begin{figure}
\hspace*{-0.9cm}\vspace*{-0.2cm}\includegraphics[scale=0.43]{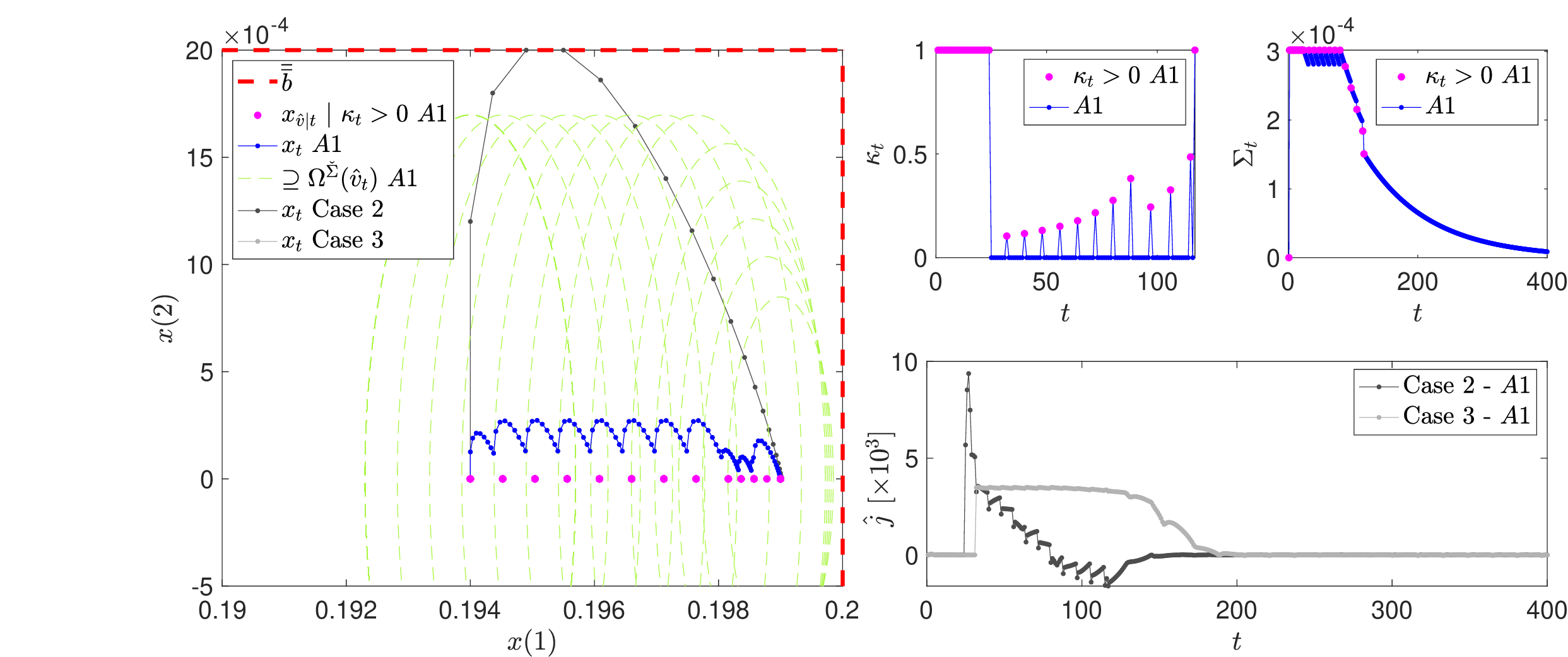}
\label{1}
\vspace*{2mm}
\caption{The left plot shows the state trajectory of the inexact plant-optimizer system (blue), and cases with no suboptimality (grey). Case 3 is nearly identical to the A1 case. In the lower right plot, relative ADMM iterations at termination are shown. Average iterations are: 311, 550, 1380 for A1, Case 2, Case 3 respectively. Wall times (ADMM solve time + Online Procedure time (in the case of A1)) are 6.10, 7.64, 19.14 sec for A1, Case 2, Case 3 respectively. The upper right plots show the governor parameter and constraint tightening parameter for A1. Tunable parameters and Lipschitz constants are as follows. $N = 3, \ R = 0.001, \ Q = I_{(1,1)}, \ \rho = 0.3, \ \gamma = 0.999, \ \phi = 1, \ \pi_1 = 0.99, \ \pi_2 = 0.000001, \ \beta_w = 200, \ \beta{\chi} = 3, \ \overline{\Sigma} = 0.1, \ \underline{\Sigma} = 0.0001, \ \omega = 1/300, \ \overline{\Lambda} = 90, \underline{\Lambda} = \underline{\epsilon}, \ \sigma = 0.4$.}
\end{figure}

\section{Conclusion}
This paper addressed feasibility and convergence for the computationally governed suboptimal LQ-MPC in which reference command modification is combined with constraint tightening to reduce the computational burden.  Our approach facilitates  early optimization algorithm termination while protecting against constraint violations due to the effects of inexact computations.  The approach is grounded in properties of ADMM-based optimization algorithms but it can also be exploited with other optimization algorithms that have appropriate properties. In order to ensure strict guarantees, conservatism is built into the procedure; less conservative approaches will be considered in future work.

\bibliography{admm}

\appendix
\bmsection{Proof of Lemma 1}

We consider $ x_{t+1}^\circ = Ax_{t} + Bu^{*0}(x_t,v)$, and $v_t = v \ \forall t$. 
 
 To begin, we use the stage cost term in (23) and relation (18) to write the inequality, $||\tilde{x}_t||_Q^2 \geq \lambda_{min}(Q) ||\tilde{x}_t||^2 \geq (\lambda_{min}(Q)/\beta_{\chi}^2)(\psi^0({\theta_t}))^2$, (Equation 102 \cite{limited_communication_data_rates}) which, plugging into (23), results in 
\begin{equation}
    \psi^0(x^{\circ}_{t+1},v) \leq \sqrt{1 - \frac{\lambda_{min}(Q)}{\beta_{\chi}^2}} \psi^0(\theta_t)
\end{equation}

We next use (18) to write the inequality $\big| \psi^0(x^\circ_{t+1},v) - \psi^0(x_{t+1},v)\big|^2 = \big| ||\tilde{\chi}^{*0}(x^{\circ}_{t+1},v)||_{H^\chi} - ||\tilde{\chi}^{*0}(x_{t+1},v)||_{H^\chi} \big|^2 \leq \beta_{\chi}^2||x^{\circ}_{t+1} - x_{t+1}||^2 \leq \overline{\lambda}(B) \beta_{\chi}^2||\hat{\tilde{u}}^{\Sigma_t}(\theta_t) - \tilde{u}^{*0}(\theta_t)||^2$. 

Then from 
$||\hat{\tilde{u}}^{\Sigma_t}(\theta_t) - \tilde{u}^{*0}(\theta_t)|| \leq
||\tilde{u}^{*\Sigma_t}(\theta_t) - \hat{\tilde{u}}^{\Sigma_t}(\theta_t)|| + 
||\tilde{u}^{*\Sigma_t}(\theta_t) - \tilde{u}^{*0}(\theta_t)||$, we make use Algorithm 1, Line 5 to observe $||\tilde{u}^{*\Sigma_t}(\theta_t) - \hat{\tilde{u}}^{\Sigma_t}(\theta_t)|| \leq \Sigma_t$, and of the Lipshitz relation (19) to observe
$||\tilde{u}^{*\Sigma_t}(\theta_t) - \tilde{u}^{*0}(\theta_t)|| \leq \phi\Sigma_t$.

Then we have
\begin{equation}
    ||\hat{\tilde{u}}^{\Sigma_t}(\theta_t) - \tilde{u}^{*0}(\theta_t)|| \leq (\phi + 1)\Sigma_t,
\end{equation}
and
\begin{equation}
    \big| \psi^0(x^\circ_{t+1},v) - \psi^0(x_{t+1},v)\big| \leq \sqrt{\overline{\lambda}(B)} \beta_{\chi} (\phi + 1)\Sigma_t.
\end{equation}

Then we use (A1) and (A3) in
$\psi^0(\theta_{t+1}) \leq \big| \psi^0(x^\circ_{t+1},v) - \psi^0(x_{t+1},v)\big| + \psi^0(x^\circ_{t+1},v)$  
which yields (27), where
\begin{gather}
\alpha_1 = 1 - \sqrt{1 - \frac{\lambda_{min}(Q)}{\beta_{\chi}^2}}, \\
\zeta_1 = \sqrt{\overline{\lambda}(B)} \beta_{\chi} (\phi + 1).    
\end{gather}

Moving on to (28), we adopt (Equation 34 \cite{jordan_limon_ROA}, Equation 30 \cite{main_suboptimal}),
\begin{equation*}
    ||\theta_{t+1} - \theta_t|| \leq \frac{\overline{\lambda}(A - I) + \overline{\lambda}(B)}{\sqrt{\lambda_{min}(H^{\chi})}} \psi^0(\theta_t) + \overline{\lambda}(B)|| \ ||\hat{\tilde{u}}^{\Sigma_t}(\theta_t) - \tilde{u}^{*0}(\theta_t)||.
\end{equation*}
We bound (27) as the following
\begin{equation}
\Sigma_{t+1} \leq \pi_1\Sigma_t + \pi_2\frac{\overline{\lambda}(A - I) + \overline{\lambda}(B)}{\sqrt{\lambda_{min}(H^{\chi})}} \psi^0(\theta_t) + \pi_2\overline{\lambda}(B) (\phi + 1) \Sigma_t
\end{equation}
where we made use of (A2). This yields (29), where 
\begin{gather}
\alpha_2 = 1-(\pi_1 + \pi_2\overline{\lambda}(B)(\phi + 1)), \\
\zeta_2 = \pi_2\frac{\overline{\lambda}(A - I) + \overline{\lambda}(B)}{\sqrt{\lambda_{min}(H^{\chi})}}.    
\end{gather}

\end{document}